%% file: main.tex
\documentclass[11pt]{article}
\usepackage[utf8]{inputenc}
\usepackage[T1]{fontenc}
\usepackage{amsmath,amssymb,lmodern,amsthm}
\usepackage{mathrsfs}
\usepackage{mathtools}
\usepackage[linktocpage=true]{hyperref}
\usepackage{bbm}
\usepackage{tikz}
\usepackage{cleveref}
\usepackage{tikz-cd}
\usepackage{comment}
\usepackage{adjustbox}
\usepackage{enumerate}
\usepackage{microtype}
\usepackage[backend=biber,
style=alphabetic,
sorting=nyt,eprint=false, isbn=false, doi=false, giveninits=true, date=year, hyperref
]{biblatex}
\hypersetup{
    breaklinks=true
}

\DeclareMathOperator{\Id}{\operatorname{Id}}
\DeclareMathOperator{\cDGA}{\operatorname{coDGA}^{\operatorname{conil}}}
\DeclareMathOperator{\icDGA}{\underline{\operatorname{coDGA}}^{\operatorname{conil}}}
\DeclareMathOperator{\iVec}{\underline{\operatorname{DGVec}}}
\DeclareMathOperator{\eDGA}{\overline{\operatorname{DGA}}_0}
\DeclareMathOperator{\DGA}{\operatorname{DGA}}
\DeclareMathOperator{\DGLA}{\operatorname{DGLA}}
\DeclareMathOperator{\cCDGA}{\operatorname{coCDGA}^{\operatorname{conil}}}
\DeclareMathOperator{\cDGLA}{\operatorname{coDGLA}^{\operatorname{conil}}}
\DeclareMathOperator{\icCDGA}{\underline{\operatorname{coCDGA}}^{\operatorname{conil}}}
\DeclareMathOperator{\DGVec}{\operatorname{DGVec}}
\theoremstyle{plain}
\newtheorem{thm}{Theorem}[section]
\newtheorem{prop}[thm]{Proposition}
\newtheorem{cor}[thm]{Corollary}
\newtheorem{lemma}[thm]{Lemma}

\theoremstyle{definition}
\newtheorem{definition}[thm]{Definition}
\newtheorem{example}[thm]{Example}
\theoremstyle{remark}
\newtheorem{remark}[thm]{Remark}
\tikzcdset{column sep/etiny/.initial=-0.0cm}

\title{Enriched Koszul Duality}
\author{Björn Eurenius}
\date{}

\begin{document}

\maketitle

\begin{abstract}
    We show that the category of non-counital conilpotent dg-coalgebras and the category of non-unital dg-algebras carry model structures compatible with their closed non-unital monoidal and closed non-unital module category structures respectively. Furthermore, we show that the Quillen equivalence between these two categories extends to a non-unital module category Quillen equivalence, i.e. providing an enriched form of Koszul duality.
\end{abstract}

\tableofcontents

\input{Introduction}

\input{coAlg.tex}

\input{KoszulandMC}

\input{Semi}

\input{Semimodel}

\input{Modulestructure}

\input{DGAadjunctions}

\input{Measurings}

\input{McDGA}

\input{MDGA}

\input{DGLie}

\input{Koszulduality}

\printbibliography
\noindent
Department of Mathematics and Statistics, Lancaster University,\\
Lancaster, LA1 4YF, UK\\
E-mail address: b.eurenius@lancaster.ac.uk

\end{document}

%% file: Introduction.tex
\section{Introduction}
Koszul duality appears in several areas of algebra, topology and geometry where its origin can be traced back to the inception of rational homotopy theory as developed by Quillen in \cite{Quillen1969}. We will mainly be working in the context of \emph{associative dg Koszul duality} which, in its modern formulation, can be expressed as Quillen equivalence between the category of non-unital differential graded algebras, $\DGA_0$, and the category of conilpotent non-counital differential graded coalgebras, $\cDGA$, over some field $k$. This case was initially shown in \cite{Lefevre2003} and further developed in \cite{Positselski2011}. A similar result in what we will refer to as \emph{com-Lie dg Koszul duality}, previously achieved in \cite{Hinich2001}, can be expressed as a Quillen equivalence between the category of differential graded Lie algebras, $\DGLA$, and the category of conilpotent non-counital cocommutative coalgebras, $\cCDGA$, over some field $k$ of characteristic zero. For a general survey of these results as well as further reading on Koszul duality we refer the reader to \cite{Positselski2022}. 

On a somewhat different track, it has been shown in \cite{Anel2013} that the category $\DGA_0$ is enriched, tensored, and cotensored over the closed symmetric monoidal category of non-counital differential graded coalgebras $(\text{coDGA}_0,\otimes)$ equipped with the ordinary tensor product. This was further extended to the operadic setting in \cite{LeGrignou2019}. Our main aim with this paper is to provide a strengthening of Koszul duality that respects this enrichment of algebras over coalgebras. However in the case of associative Koszul duality we are dealing with the category of conilpotent coalgebras, $\cDGA$, which does not have a monoidal unit under the ordinary tensor product making it into a semi-monoidal category. 

Motivated by the lack of a unit, we introduce the notion of semi-module categories over a semi-monoidal category. Taking this further in the homotopical direction, we introduce semi-monoidal model categories and semi-module model categories analogous to their unital counterparts. Remembering that a closed module category is precisely the same thing as a tensored and cotensored enriched category, over the same monoidal category, this also allows us to speak about what we will refer to as a semi-enrichment over a semi-monoidal category. While a priori a weaker concept than enriched category theory, it nevertheless puts structure limitations on the categories in question. Furthermore, while our main interest as well as initial motivation is that of Koszul duality, one quickly finds that semi-monoidal model categories that are not monoidal do appear naturally in homotopical algebra. 

Having established the framework of semi-monoidal- and semi-module- categories as well as their model categorical analogues, we proceed to show that $\DGA_0$ can be given a closed semi-module category structure over the semi-monoidal category $(\cDGA,\otimes)$ using the same procedure as in \cite{Anel2013}. Furthermore, we show that their semi-module structures are compatible with their respective model category structures as well as with the Quillen equivalence that is associative Koszul duality. Proceeding similarly in the case of com-Lie Koszul duality, one obtains the corresponding result in that setting. Specifically, our main result in the associative setting is
\begin{thm}\label{thm1}
    The category $(\cDGA,\otimes,\icDGA)$ is a semi-monoidal model category and $(\DGA_0,\rhd,\overline{\DGA}_0,\{-,-\})$ is a semi-module model category over $\cDGA$. Furthermore, the Quillen equivalence
    \begin{equation*}
    \begin{tikzcd}
        \cDGA \arrow[r,shift left=1.5ex,"\Omega"] \arrow[r,phantom,"\perp"] & \DGA_0 \arrow[l,shift left=1.5ex,"B"]
    \end{tikzcd}
    \end{equation*}
    upgrades to a $\cDGA$-module Quillen equivalence. 
\end{thm}
Here we used the notation $\icDGA$ for the internal hom of $\cDGA$, while $\rhd$, $\eDGA$, and $\{-,-\}$ corresponds to the tensoring, enrichment, and cotensoring functors of the closed semi-module structure of $\DGA_0$ respectively. In particular the cotensoring $\{-,-\}$ is the convolution algebra functor.

    The corresponding result in the com-Lie context is the following.
    \begin{thm}\label{thmLie}
    The category $(\cCDGA,\otimes,\icCDGA)$ is a semi-monoidal model category and $(\DGLA,\rhd,\overline{\DGLA},\{-,-\})$ is a semi-module model category over $\cCDGA$. Furthermore the Quillen equivalence
    \begin{equation*}
    \begin{tikzcd}
        \cCDGA \arrow[r,shift left=1.5ex,"\Omega"] \arrow[r,phantom,"\perp"] & \DGLA \arrow[l,shift left=1.5ex,"B"]
    \end{tikzcd}
    \end{equation*}
    upgrades to a $\cCDGA$-module Quillen equivalence. 
\end{thm}
Here we used the notation $\icCDGA$ for the internal hom functor of
$\cCDGA$, while $\rhd$, $\overline{\DGLA}$, and $\{-,-\}$ corresponds to the tensoring, enrichment, and cotensoring functors of the closed semi-module structure of $\DGLA$ respectively. In particular the cotensoring $\{-,-\}$ is the convolution algebra functor.

Finally, let us note that similar results in the context of dg-categories have been achieved in \cite{2022Holstein}, where they provide a homotopical enrichment of dg-categories over pointed coalgebras. Note that the category of pointed algebras is monoidal, as opposed to just semi-monoidal, and as such they provide the category of dg-categories with an enriched category structure. 

\subsection*{Acknowledgements}
The author thanks Andrey Lazarev for numerous helpful discussions and suggestions while preparing this paper. The author also thanks Ai Guan and Lauren Hindmarch for providing feedback on earlier drafts.

%% file: coAlg.tex
\subsection{Conilpotent coalgebras}
    We will here recall some facts about dg-coalgebras and conilpotent dg-coalgebras in particular. We will be working with the category of non-counital coassociative dg-coalgebras, $\text{coDGA}_0$, over some field $k$. Explicitly, this means that we require that the comultiplication $\Delta:C \to C\otimes C$ satisfies coassociativity, i.e.
\begin{equation*}
    (\Delta \otimes \Id)\circ \Delta = (\Id \otimes \Delta)\circ \Delta,
\end{equation*}
and that the differential $d:C \to C$ is a coderivation, i.e. satisfying
\begin{equation*}
    \Delta \circ d = (\Id \otimes d+d\otimes \Id) \circ \Delta.
\end{equation*}
However we will make no demands on the existence of a counit morphism. If a dg-coalgebra $(C,\Delta)$ satisfies that $\Delta = B_{C,C} \circ \Delta$, where $B_{C,C}:x\otimes y\mapsto (-1)^{|x||y|}y\otimes x$ is the braiding morphism of dg-vector spaces, we say that $C$ is cocommutative. As all our coalgebras will be differential graded, we will from now take coalgebra to implicitly mean dg-coalgebra.

\begin{definition}
Let $(C,\Delta)$ be a non-counital coalgebra. We say that an element $x\in C$ is conilpotent if there exists some $n$ such that $\Delta^n(x) = 0$. If all elements of $C$ are conilpotent we say that $C$ is conilpotent.
\end{definition}
We will denote the category of conilpotent dg-coalgebras by $\cDGA$. 

\begin{definition}
    Let $(C,\Delta,d)$ be a non-counital coalgebra. We say that an element $c\in C$ is an atom if $\Delta(c) = c\otimes c$ and $dc = 0$.
\end{definition}
The set of atoms of a non-counital coalgebra $C$ is in one-to-one correspondence with coalgebra morphisms $k\to C$ from the monoidal unit. We note that a conilpotent coalgebra $C$ has exactly one atom $0\in C$. In particular, if $C$ is conilpotent, the zero morphism is the only morphism $k\to C$.

\begin{prop}\label{propconilclosedundertensor}
    Let $(C,\Delta_C)$ be a conilpotent coalgebra and $(D,\Delta_D)$ an arbitrary non-unital coalgebra. Then $(C\otimes D,\Delta_{C\otimes D})$ is also conilpotent.
\end{prop}
\begin{proof}
    Let $c\otimes d$ be a pure tensor in $C\otimes D$. By assumption there exists some $n>0$ such that $\Delta_C^n(c) = 0$. Then we have
    \begin{equation*}
        \Delta_{C\otimes D}^n(c\otimes d) := (\Delta_C \otimes \Delta_D)^n(c\otimes d) = \Delta^n_C(c) \otimes \Delta^n_D(d)=0,
    \end{equation*}
    where we have not done any reordering of the terms as it only affects the sign.
\end{proof}

    \begin{definition}
        Let $(C,d_C,\Delta_C)\in \cDGA$ be a conilpotent dg-coalgebra. An admissible filtration of $C$ is an increasing filtration $F$ starting at $F_1$ compatible with the differential and comultiplication, meaning that
        \begin{equation*}
            \begin{aligned}
                &d(F_n) \subset F_n, \text{ and}\\
                &\Delta(F_n) \subset \bigoplus_{k=1}^{n-1} F_{n-k}\otimes F_k.
            \end{aligned}
        \end{equation*}
    \end{definition}
    
     \begin{example}
        For any conilpotent dg-coalgebra $(C,\Delta)$ there is a canonical admissible filtration, known as the coradical filtration, given by $F_n = \ker {\Delta}^n$. 
    \end{example}
    \begin{definition}
        We say that a morphism  $f:C\to D$ in $\cDGA$ is a filtered quasi-isomorphism if there exists admissible filtrations $\mathcal{F}_C$ and $\mathcal{F}_D$ of $C$ and $D$ respectively such that the induced morphism of the associated graded complexes $\operatorname{gr} f:\operatorname{gr}(\mathcal{F}_C) \to \operatorname{gr}(\mathcal{F}_D)$ is a quasi-isomorphism in each degree. 
    \end{definition}

Several of our proofs rely on the existence of cofree objects in the categories of $\text{coDGA}_0$ and $\cDGA$. The cofree functor $\check{T}_0$ is defined as the right adjoint to the forgetful functor to the category of differential graded vector spaces, $\DGVec$, i.e.
\begin{equation*}
        \begin{tikzcd}
        \operatorname{coDGA}_0 \arrow[r,shift left=1.5ex,"U"] \arrow[r,phantom,"\perp"] & \DGVec \arrow[l,shift left=1.5ex,"\check{T}_0"].
    \end{tikzcd}
\end{equation*}
The existence of this adjunction was shown in \cite{Block1985}.
We will say that a coalgebra of the form $\check{T}_0V$ for some $V\in \DGVec$ is cofree.
Analogously the conilpotent cofree functor $T^{\operatorname{co}}$ is defined to satisfy the adjunction
\begin{equation*}
        \begin{tikzcd}
        \cDGA \arrow[r,shift left=1.5ex,"U"] \arrow[r,phantom,"\perp"] & \DGVec \arrow[l,shift left=1.5ex,"T^{\operatorname{co}}"],
    \end{tikzcd}
\end{equation*}
and we say that a conilpotent coalgebra of the form $T^{\operatorname{co}}V$ for some $V\in \DGVec$ is conilpotent cofree. Conilpotent cofree coalgebras are also known as tensor coalgebras as they are constructed analogously to the tensor algebra but instead given the cut comultiplication. 

There is also an adjunction,
\begin{equation*}
        \begin{tikzcd}
        \cDGA \arrow[r,shift left=1.5ex,"\iota"] \arrow[r,phantom,"\perp"] & \operatorname{coDGA}_0 \arrow[l,shift left=1.5ex,"{R^{\operatorname{co}}}"],
    \end{tikzcd}
\end{equation*}
between the inclusion functor $\iota:\cDGA \to \text{coDGA}_0$ and the conilpotent radical functor $R^{\operatorname{co}}:\text{coDGA}_0 \to \cDGA$. The latter is defined by taking a coalgebra $C$ to the subcoalgebra consisting of conilpotent elements of $C$. The adjunction follows from that the image of a conilpotent element under a coalgebra morphisms is also conilpotent. As a consequence of this adjunction we have that $T^{\operatorname{co}} \cong R^{\operatorname{co}}\check{T}_0$.

When working with the category of cocommutative conilpotent non-counital dg-coalgebras $\cCDGA$ we will also make use of the adjunction
\begin{equation*}
        \begin{tikzcd}
        \cCDGA \arrow[r,shift left=1.5ex,"U"] \arrow[r,phantom,"\perp"] & \DGVec \arrow[l,shift left=1.5ex,"S^{\operatorname{co}}"].
    \end{tikzcd}
\end{equation*}
Here $S^{\operatorname{co}}$ denotes the cocommutative conilpotent cofree functor, which takes a dg-vector space $V$ to the coabelianization of the conilpotent cofree algebra over $V$. We say that a cocommutative conilpotent coalgebra of the form $S^{\operatorname{co}}V$ for $V\in \DGVec$ is cocommutative conilpotent cofree.

%% file: KoszulandMC.tex
\subsection{Maurer-Cartan elements and Koszul duality}

\begin{definition}
    Let $(A,d_A)$ be a non-unital dg-algebra. We say an element $a\in A$ of degree $-1$ is a Maurer-Cartan element if it satisfies the Maurer-Cartan equation
    \begin{equation*}
        d_Aa + a^2 = 0.
    \end{equation*}
    We denote the set of Maurer-Cartan elements of $A$ by $\operatorname{MC}(A)$.
\end{definition}

    We define the universal Maurer-Cartan algebra $\mathbf{mc}$ as the non-unital free algebra $T_0\langle x\rangle$, where  $\langle x\rangle$ denotes the vector space generated by the variable $x$ of degree $-1$, and given the differential induced from $d:x\mapsto -x^2$. Noting that any Maurer-Cartan element $a\in A$ corresponds to the unique morphism in $\DGA(\mathbf{mc},A)$ generated by $x\mapsto a$ we get the following proposition.
\begin{prop}
    The functor $\operatorname{MC}:\DGA_0 \to \operatorname{Set}$ taking a dg-algebra to its set of Maurer Cartan elements is representable by the universal Maurer-Cartan element $\mathbf{mc}$.
\end{prop}

Similarly in the case of dg-Lie algebras we have the following definition.
\begin{definition}
Let $(\mathfrak{g},[-,-],d)$ be a dg-Lie algebra. We say an element $x \in \mathfrak{g}$ of degree $-1$ is a Maurer-Cartan element if it satisfies the Maurer-Cartan equation
\begin{equation*}
    dx + \frac{1}{2}[x,x] = 0.
\end{equation*}
We denote the set of Maurer-Cartan elements of $\mathfrak{g}$ by $\operatorname{MC}_{\operatorname{Lie}}(\mathfrak{g})$.
\end{definition}
Similar to the associative case, we define the universal Maurer-Cartan Lie algebra $\mathbf{mc}_{\operatorname{Lie}}$ as the free Lie algebra $T_{\operatorname{Lie}} \langle x\rangle$, where  $\langle x\rangle$ denotes the vector space generated by the variable $x$ of degree $-1$, and given the differential induced from $d:x \mapsto -\frac{1}{2}[x,x]$. Noting that any Maurer-Cartan element $a\in \mathfrak{g}$ corresponds to the unique morphism in $\DGLA(\mathbf{mc}_{\operatorname{Lie}},\mathfrak{g})$ generated by $x\mapsto a$ we get the following proposition.
\begin{prop}
The functor $\operatorname{MC}_{\operatorname{Lie}}:\DGLA \to \operatorname{Set}$ taking a dg-Lie algebra to its set of Maurer-Cartan elements is representable by the universal Maurer-Cartan element $\mathbf{mc}_{\operatorname{Lie}}$.
\end{prop}

Next we remind ourselves of the convolution coalgebra construction. 
    \begin{definition}
        Let $(C,\Delta_C,d_C)$ be a dg-coalgebra and $(A,\mu_A,d_A)$ a dg-algebra. Then the internal hom of dg-vector space, $\iVec(C,A)$ has the structure of a dg-algebra with multiplication defined as the convolution product
        \begin{equation*}
            f*g := C \xrightarrow{\Delta_C} C\otimes C \xrightarrow{f\otimes g} A\otimes A \xrightarrow{\mu_A} A,
        \end{equation*}
        and differential
        \begin{equation*}
            df := d_Af -(-1)^{|f|}fd_C.
        \end{equation*}
        We will refer to this construction as the convolution algebra of $C$ into $A$ and denote it by $\{C,A\}$.
    \end{definition}
    We caution the reader that our choice of notation for the convolution algebra conflicts with the notation used in \cite{Anel2013}. They use the notation $[-,-]$ for the convolution algebra while $\{-,-\}$ instead denotes the measuring coalgebra construction. 
    \begin{prop}\label{convcComLie}
    Let $C$ be a cocommutative non-unital coassociative dg-coalgebra and $\mathfrak{g}$ a dg-Lie algebra. Then the convolution algebra $\{C,\mathfrak{g}\}$, defined as above, takes the form of a dg-Lie algebra.
    \end{prop}
\begin{proof}
    We check that $\{C,\mathfrak{g}\}$ is a Lie algebra. As $C$ is cocommutative, i.e. satisfying that $\Delta(c) = c^{(1)}\otimes c^{(1)}$, antisymmetry is satisfied by $[f,f]\circ \Delta_C(c) = 0$ for all $f\in \{C,\mathfrak{g}\}$. We also have the Jacobi identity as for all $f,g,h\in \{C,A\}$ we have that
    \begin{equation*}
        \left(f*(g*h) + g*(h*f) + h*(f*g)\right)\Delta^2(c) = 0,
    \end{equation*}
    since by cocommutativity it follows that $\Delta^2(c) = c^{(1)}\otimes c^{(1)} \otimes c^{(1)}$.
\end{proof}

    We will from now on consider the convolution algebra functor restricted to the category of conilpotent coalgebras
    \begin{equation*}
        \{-,-\}:\left(\cDGA\right)^{\operatorname{op}} \times \DGA_0 \to \DGA_0.
    \end{equation*}

    Let us now briefly recall Koszul duality and the bar and cobar constructions. We refer the reader to \cite{Positselski2011} for the proofs and further background. We will adopt the convention of using homological grading throughout. For an algebra $(A,m,d_A) \in \DGA_0$, we define the bar construction $BA$ as a graded coalgebra to be the conilpotent cofree coalgebra $T^{\operatorname{co}}(\Sigma^{-1} A)$  where $\Sigma^{-1}A$ is the shifted complex with $(\Sigma^{-1}A)_n := A_{n-1}$. The bar construction $BA$ is then given the differential induced by $d_A + m$.

    Conversely, given a coalgebra $(C,\Delta,d_C)$ we define the cobar construction $\Omega C$ as a graded algebra to be the free algebra $T\Sigma C$ with differential induced from $d_C + \Delta$. 

    The bar and cobar functors can be shown to be adjoint by
    \begin{equation*}
        \DGA_0(\Omega C, A) \cong \operatorname{MC}(\{C,D\})\cong \cDGA(C,BA).
    \end{equation*}
    Furthermore when considering categories $\DGA_0$ and $\cDGA$ with their standard model structures the above adjunction upgrades to a Quillen equivalence. Explicitly the model structure on $\DGA_0$ is given as follows. We say a morphism in $\DGA_0$ is a
    \begin{enumerate}[i)]
        \item weak equivalence if it is a quasi-isomorphism,
        \item fibration if it is surjective,
        \item cofibration if it has the left lifting property with respect to all acyclic fibrations.
    \end{enumerate}
    The category of $\cDGA$ admits the left transferred model structure over the above adjunction. We say that a morphism in $\cDGA$ is a
    \begin{enumerate}[i)]
        \item weak equivalence if it belongs to the minimal class of morphisms generated by filtered quasi-isomorphism under the 2 out of 3 property,
        \item cofibration if it is injective,
        \item fibration if it has the right lifting property with respect to all acyclic cofibrations.
    \end{enumerate}
    For a proof of the existence of these model structures as well as the Quillen equivalence we refer the reader to \cite{Positselski2011}.  

    The story in the com-Lie case of Koszul duality is very similar but we will in this case require that the ground field $k$ has characteristic zero.  This is needed for the existence of the model structure on $\DGLA$ and $\cCDGA$ as shown in \cite{Hinich2001}. As in the associative case we say that a morphism in $\DGLA$ is a
    \begin{enumerate}[i)]
        \item weak equivalence if it is a quasi-isomorphism,
        \item fibration if it is surjective,
        \item cofibration if it has the left lifting property with respect to all acyclic fibrations.
    \end{enumerate}
    Similarly we say that a morphism in $\cCDGA$ is a
    \begin{enumerate}[i)]
        \item weak equivalence if it belongs to the minimal class of morphisms generated by filtered quasi-isomorphism under the 2 out of 3 property,
        \item cofibration if it is injective,
        \item fibration if it has the right lifting property with respect to all acyclic cofibrations.
    \end{enumerate}

    For a Lie-algebra $(\mathfrak{g},[-,-],d_{\mathfrak{g}})\in \DGLA$ we define the bar construction $B \mathfrak{g}$ to be $S^{co}(\Sigma^{-1}\mathfrak{g})$ with differential induced from $d_{\mathfrak{g}} + [-,-]$. In the other direction for a cocommutative coalgebra $(C,\Delta,d_C) \in \cCDGA$ we define the cobar construction $\Omega C$ as $T_{\operatorname{Lie}}(\Sigma C)$ with differential induced by $d_C + \Delta$. These functors are Quillen equivalent by the adjunction
\begin{equation*}
    \DGLA(\Omega C, \mathfrak{g}) \cong \operatorname{MC}_{\operatorname{Lie}}(\{C,\mathfrak{g}\})\cong \cCDGA(C,B\mathfrak{g}),
\end{equation*}
as shown in \cite{Hinich2001}. Note in particular that $\{C,A\}$ has the structure of a dg Lie-algebra by \Cref{convcComLie}.

%% file: Semi.tex
\section{Semi-monoidal categories, semi-module categories, and semi-enrichments}
Categories that are monoidal except missing a unit, known as semi-monoidal, semi-groupal or non-unital monoidal in the literature have previously been studied in e.g. \cite{Kock2008}, \cite{Abuhlail2013}, and \cite{Lu2019}. We will introduce the notion of semi-module categories over a semi-monoidal category, fully analogous to the definition in the unital case. Working in the model category setting this will lead us to the definition of semi-monoidal model categories and semi-module model categories. 

We will take the definition of (symmetric) semi-monoidal categories, semi-monoidal functors, and semi-monoidal natural transformations to be fully analogous to the monoidal ones by dropping the unit and unit axioms at every step. Similarly, we take take the definitions of semi-modules, semi-module functors and semi-module natural transformations to be those found in appendix B of \cite{Anel2013} or chapter 4 in \cite{Hovey1999} by dropping the unit axioms. When working over a semi-monoidal category $\mathcal{V}$ we will commonly use the terminology $\mathcal{V}$-module to mean a semi-module over $\mathcal{V}$ etc.

\begin{definition}
    A semi-monoidal category $(\mathcal{V},\otimes,a)$ consists of a category $\mathcal{V}$ together with a functor
    \begin{equation*}
        \otimes:\mathcal{V}\times \mathcal{V} \to \mathcal{V},
    \end{equation*}
    and a natural isomorphism
    \begin{equation*}
        a:(X \otimes Y) \otimes Z \to X \otimes (Y\otimes Z),
    \end{equation*}
    satisfying the coherence diagram
    \begin{equation*}
                \begin{tikzcd}
                    &((X\otimes Y) \otimes Z) \otimes W \arrow[dl,"a_{X,Y,Z}\otimes \Id_W"'] \arrow[dr,"a_{X\otimes Y,Z,W}"]\\
                    (X\otimes (Y\otimes Z))\otimes W \arrow[d,"a_{X,Y\otimes Z,W}"'] && (X\otimes Y)\otimes (Z\otimes W) \arrow[d,"a_{X,Y,Z\otimes W}"]\\
                    X \otimes ((Y\otimes Z)\otimes W) \arrow[rr,"\Id_X\otimes a_{Y,Z,W}"'] && X\otimes (Y\otimes (Z\otimes W))
                \end{tikzcd}
            \end{equation*}
            for all $X,Y,Z,W \in \mathcal{V}$.
\end{definition}

\begin{definition}
    A symmetric semi-monoidal category $(\mathcal{V},\otimes, a,c)$ is a semi-monoidal category $(\mathcal{V},\otimes,a)$ together with a natural isomorphism $c:X\otimes Y \to Y\otimes X$ satisfying that
     \begin{equation*}
            \begin{tikzcd}
                X \otimes Y \arrow[r,"c_{X,Y}"] \arrow[dr,"\Id_{X\otimes Y}"'] & Y \otimes X \arrow[d,"c_{Y,X}"]\\
                & X\otimes Y
            \end{tikzcd}
        \end{equation*}
        and
    \begin{equation*}
            \begin{tikzcd}
                (X \otimes Y) \otimes Z \arrow[r,"a_{X,Y,Z}"] \arrow[d,"c_{X,Y}\otimes \Id_Z"'] & X \otimes (Y\otimes Z) \arrow[r,"c_{X,Y\otimes Z}"] & (Y\otimes Z)\otimes X \arrow[d,"a_{Y,Z,X}"]\\
                (Y\otimes X) \otimes Z \arrow[r,"a_{Y,X,Z}"'] & Y\otimes(X\otimes Z) \arrow[r,"\Id_Y \otimes c_{X,Z}"'] & Y \otimes (Z \otimes X).
            \end{tikzcd}
        \end{equation*}
        commute for all $X,Y,Z \in \mathcal{V}$.
\end{definition}

\begin{definition}
    We say that a semi-monoidal category $(\mathcal{V},\otimes)$ is (left) closed if
    there exists a functor
    \begin{equation*}
        \underline{\mathcal{V}}:\mathcal{V}^{\operatorname{op}} \times \mathcal{V} \to \mathcal{V},
    \end{equation*}
    such that $\underline{\mathcal{V}}(Y,-)$ is right adjoint to $-\otimes Y$. The functor $\underline{\mathcal{V}}$ is known as the internal hom functor.
\end{definition}

\begin{definition}
    Let $(\mathcal{V},\otimes,a)$ and $(\mathcal{V},\otimes',a')$ be two semi-monoidal categories. Then a (strong) semi-monoidal functor $(F,m)$ from $\mathcal{V}$ to $\mathcal{V}'$ consists of a functor $F:\mathcal{V}\to \mathcal{V}'$ and a natural isomorphism
    \begin{equation*}
        m:F(X) \otimes' F(Y) \to F(X\otimes Y),
    \end{equation*}
    satisfying the coherence diagram
    \begin{equation*}
            \begin{tikzcd}
                (F(X)\otimes' F(Y))\otimes' F(Z) \arrow[r,"a'"] \arrow[d,"m\otimes' \Id_{F(Z)}"'] & F(X) \otimes' (F(Y)\otimes' F(Z)) \arrow[d,"\Id_{F(X)} \otimes' m"]\\
                F(X\otimes Y) \otimes' F(Z) \arrow[d,"m"'] & F(X) \otimes' F(Y\otimes Z) \arrow[d,"m"]\\
                F((X\otimes Y)\otimes Z) \arrow[r,"F(a)"] & F(X \otimes (Y\otimes Z))
            \end{tikzcd}
        \end{equation*}
        for all $X,Y,Z \in \mathcal{V}$.
\end{definition}

\begin{definition}
    Let $(F,m)$ and $(F',m')$ be two semi-monoidal functors from $(\mathcal{V},\otimes)$ to $(\mathcal{V}',\otimes')$. A semi-monoidal natural transformation $\eta:F\to F'$ is a natural transformation such that
    \begin{equation*}
        \begin{tikzcd}
            F(X) \otimes' F(Y) \arrow[r,"\eta \otimes' \eta"] \arrow[d,"m"'] & F'(X) \otimes' F'(Y) \arrow[d,"m'"]\\
            F(X\otimes Y) \arrow[r,"\eta"] & F'(X\otimes Y)
        \end{tikzcd}
    \end{equation*}
    commutes for all $X,Y\in \mathcal{V}$.
\end{definition}

We could at this point further proceed with the theory of semi-monoidal categories, by defining semi-monoidal adjunctions and semi-monoidal equivalences etc. As we will not explicitly need them, we instead proceed with the definition of semi-module categories.
\begin{definition}
    Let $(\mathcal{V},\otimes,a)$ be a symmetric semi-monoidal category. A (left) $\mathcal{V}$-module is a category $\mathcal{C}$ together with a functor
    \begin{equation*}
        - \rhd -:\mathcal{V} \times \mathcal{C} \to \mathcal{C}
    \end{equation*}
    and a natural isomorphism
    \begin{equation*}
        \alpha:(X\otimes Y)\rhd A \to X \rhd (Y \rhd A),
    \end{equation*}
    such that
    \begin{equation*}
                \begin{tikzcd}
                    &((X\otimes Y) \otimes Z) \rhd A \arrow[dl,"a"'] \arrow[dr,"\alpha"]\\
                    (X\otimes (Y\otimes Z))\rhd A \arrow[d,"\alpha"'] && (X \otimes Y) \rhd (Z\rhd A) \arrow[d,"\alpha"]\\
                    X\rhd ((Y\otimes Z)\rhd A) \arrow[rr,"\Id_X\rhd \alpha"'] && X \rhd (Y\rhd (Z\rhd A))
                \end{tikzcd}
            \end{equation*}
            commutes for all $X,Y,Z\in \mathcal{V}$ and $A\in \mathcal{C}$.
\end{definition}
A module over the opposite symmetric semi-monoidal category $(\mathcal{V}^{op},\otimes)$ is known as a $\mathcal{V}$-opmodule. We should at this point note that Mac Lane's coherence theorem holds in the case of semi-monoidal categories and the corresponding version for module categories similarly holds for semi-module categories. 

\begin{definition}
We say that a $\mathcal{V}$-module $\mathcal{C}$ is right closed if there exists a functor
\begin{equation*}
    [-,-]:\mathcal{V}^{\operatorname{op}}\times \mathcal{C} \to \mathcal{C}
\end{equation*}
such that $[X,-]$ is right adjoint to $X\rhd -$. Similarly we say that $\mathcal{C}$ is left closed if there exists a functor
\begin{equation*}
    \overline{\mathcal{C}}(-,-):\mathcal{C}^{\operatorname{op}} \times \mathcal{C} \to \mathcal{V}
\end{equation*}
such that $\overline{\mathcal{C}}(A,-)$ is right adjoint to $-\rhd A$. If $\mathcal{C}$ is both left- and right closed we say it is closed.
\end{definition}
Particularly in the closed case, we will refer to the $[-,-]$ functor as the cotensoring functor and $\overline{\mathcal{C}}(-,-)$ as the enrichment functor. Note that the cotensoring functor $[-,-]$ makes $\mathcal{C}$ into a $\mathcal{V}$-opmodule. The motivation for referring to $\overline{\mathcal{C}}(-,-)$ as an enrichment is that in the monoidal case the notion of a closed $\mathcal{V}$-module is equivalent to a tensored and cotensored $\mathcal{C}$-category.

\begin{definition}
    Let $(\mathcal{V},\otimes)$ be a symmetric semi-monoidal category and $(\mathcal{C},\rhd,\alpha)$ and $(\mathcal{D},\rhd',\alpha')$ two $\mathcal{V}$-modules. A (strong) $\mathcal{V}$-module functor $(F,m)$ from $\mathcal{C}$ to $\mathcal{D}$ consists of a functor $F:\mathcal{C} \to \mathcal{D}$ together with a natural isomorphism
    \begin{equation*}
        m:X\rhd' FA \to F(X\rhd A),
    \end{equation*}
    such that
    \begin{equation*}
        \begin{tikzcd}
            (X\otimes Y) \rhd' FA\arrow[r,"m"] \arrow[d,"\alpha'"'] & F((X\otimes Y)\rhd A) \arrow[dd,"F\alpha"]\\
            X \rhd' (Y\rhd' FA) \arrow[d,"\Id_X\rhd' m"']\\
            X\rhd' F(Y\rhd A) \arrow[r,"m"] & F(X\rhd (Y\rhd A))
        \end{tikzcd}
    \end{equation*}
    commutes for all $X,Y\in \mathcal{V}$ and $A\in \mathcal{C}$. 
\end{definition}

\begin{prop}
    Let $(\mathcal{C},\rhd)$ be a $\mathcal{V}$-module. Then for any $Y\in \mathcal{V}$ the tensoring functor $Y\rhd -$ is itself a $\mathcal{V}$-module functor. 
\end{prop}
\begin{proof}
    We take the natural isomorphism,
    \begin{equation*}
        m: X \rhd (Y \rhd A) \to Y \rhd (X\rhd A),
    \end{equation*}
    to be the composition
    \begin{equation*}
        X \rhd (Y \rhd A) \xrightarrow{\sim} (X\otimes Y) \rhd A \xrightarrow{\sim} (Y\otimes X) \rhd A \xrightarrow{\sim} Y\rhd (X\rhd A).
    \end{equation*}
\end{proof}

\begin{definition}
        Let $\mathcal{V}$ be a symmetric semi-monoidal category and $(F,m)$ and $(F',m')$ be two $\mathcal{V}$-module functors from $(\mathcal{C},\rhd)$ to $(\mathcal{D},\rhd')$. A $\mathcal{V}$-module natural transformation is a natural transformation $\eta:F\to F'$ such that
        \begin{equation*}
            \begin{tikzcd}
                    X \rhd' FA \arrow[r,"m"] \arrow[d,"\Id_X\rhd' \eta_A"'] & F(X\rhd A) \arrow[d,"\eta_{X\rhd' A}"]\\
                    X\rhd' F'A \arrow[r,"m'"] & F'(X\rhd A)
            \end{tikzcd}
        \end{equation*}
        commutes for all $X\in \mathcal{V}$ and $A\in \mathcal{C}$.
    \end{definition}
    We define the concept of semi-module adjunctions and equivalences as follows.
    \begin{definition}
        A $\mathcal{V}$-module adjunction $(F,U,\phi,m)$ is an adjunction $(F,U,\phi)$ such that $(F,m)$ is a $\mathcal{V}$-module functor. 
    \end{definition}
    \begin{definition}
        A $\mathcal{V}$-module equivalence $(F,U,\phi,m)$ is an equivalence of categories $(F,U,\phi)$ such that $(F,m)$ is a $\mathcal{V}$-module functor.
    \end{definition}
    Note that these are the notion of adjunction and equivalence in the 2-category of $\mathcal{V}$-module categories and \emph{lax} $\mathcal{V}$-module functors, i.e. where the natural transformation is not required to be an isomorphism.

     Finally we would like to mention that we can define the concept of what we call semi-enriched categories by extracting the properties of the enrichment functor in a closed semi-module category. Proceeding similarly as in the monoidal case we get the definition of such a structure. 
    \begin{definition}
    Let $(\mathcal{V},\otimes,a)$ be a symmetric semi-monoidal category. A semi-enriched $\mathcal{V}$-category is a category $\mathcal{C}$ together with a functor
    \begin{equation*}
        \overline{\mathcal{C}(-,-)}:\mathcal{C}^{\operatorname{op}} \times \mathcal{C} \to \mathcal{V}
    \end{equation*}
    and a natural transformation
    \begin{equation*}
        \mathbf{c}:\overline{\mathcal{C}}(B,C) \otimes \overline{\mathcal{C}}(A,B) \to \overline{\mathcal{C}}(A,C)
    \end{equation*}
    such that
    \begin{equation*}
        \adjustbox{scale=0.8,center}{
                \begin{tikzcd}
                    &(\overline{\mathcal{C}}(C,D) \otimes \overline{\mathcal{C}}(B,C)) \otimes \overline{\mathcal{C}}(A,B) \arrow[dl,"a"'] \arrow[dr,"\mathbf{c}\otimes \Id"]\\
                    \overline{\mathcal{C}}(C,D) \otimes (\overline{\mathcal{C}}(B,C) \otimes \overline{\mathcal{C}}(A,B)) \arrow[d,"\Id \otimes \mathbf{c}"'] && \overline{\mathcal{C}}(B,D) \otimes \overline{\mathcal{C}}(A,B) \arrow[d,"\mathbf{c}"]\\
                    \overline{\mathcal{C}}(C,D) \otimes \overline{\mathcal{C}}(A,C) \arrow[rr,"\mathbf{c}"'] && \overline{\mathcal{C}}(A,D)
                \end{tikzcd}
                }
            \end{equation*}
    commutes.
\end{definition}
Note this structure is different from that of semi-categories, i.e. categories without a unit morphism. Most notably, the underlying category is an essential part of the structure.  This is necessary as if we had taken a definition as for semi-categories, i.e. enriched categories without units, as has been studied in \cite{Moens2002} and \cite{Stubbe2005} we would have no way of recovering an underlying category. In particular, going to back to our definition, the enriched hom objects don't necessarily provide information about the hom set of the underlying category.

%% file: Semimodel.tex
\section{Semi-monoidal model categories and semi-module model categories}
We take the definitions of semi-monoidal model categories and module model categories to be analogous to the corresponding definitions with units as found in \cite{Hovey1999}.

\begin{definition}
    Let $(\mathcal{V},\otimes)$ be a closed symmetric semi-monoidal category. If furthermore $\mathcal{V}$ is a model category and the tensor functor,
    \begin{equation*}
        \otimes:\mathcal{V} \times \mathcal{V} \to \mathcal{V},
    \end{equation*}
    is a (left) Quillen bifunctor, then we say that $\mathcal{V}$ is a semi-monoidal model category.
\end{definition}

\begin{definition}
    Let $(\mathcal{V},\otimes)$ be a semi-monoidal model category and\\
    $(\mathcal{C},\rhd)$ a closed $\mathcal{V}$-module. If furthermore $\mathcal{C}$ is a model category and the tensoring functor,
    \begin{equation*}
        \rhd:\mathcal{V} \times \mathcal{C} \to \mathcal{C},
    \end{equation*}
    is a (left) Quillen bifunctor, then we say that $\mathcal{C}$ is a $\mathcal{V}$-module model category. 
\end{definition}

\begin{definition}
    Let $\mathcal{C}$ and $\mathcal{D}$ be semi-monoidal model categories. Then a semi-monoidal Quillen adjunction $(F,U,\phi,m)$ is a Quillen adjunction\\ $(F,U,\phi)$ such that $(F,m)$ is a semi-monoidal functor. 
\end{definition}

\begin{definition}
    Let $\mathcal{V}$ be a semi-monoidal model category and $\mathcal{C}$ and $\mathcal{D}$ be $\mathcal{V}$-module model categories. A $\mathcal{V}$-module Quillen adjunction $(F,U,\phi,\mu)$ is a Quillen adjunction such that $(F,\mu)$ is a $\mathcal{V}$-module functor.
\end{definition}

\begin{example}[Reduced simplicial sets]\label{examplereducedsSet}
    The category of reduced simplicial sets $\operatorname{sSet}_0$, under the smash product $\wedge$, gives an example of a semi-monoidal model category that is not monoidal. To see this, consider the coreflective adjunction
    \begin{equation*}
        \begin{tikzcd}
            \operatorname{sSet}_0 \arrow[r,shift left=1.1ex,"\iota"] \arrow[r,phantom,"\perp"] &  \operatorname{sSet}^{*/} \arrow[l,shift left=1.1ex,"\mathscr{R}"],
        \end{tikzcd}
    \end{equation*}
    where $\iota$ is the inclusion into the category of pointed simplicial sets, $\operatorname{sSet}^{*/}$, and it's adjoint $\mathscr{R}$ takes a pointed simplicial set to the subsimplicial set whose $n$-cells are those who have the marked point as $0$-cells.
    
    The category of reduced simplicial sets admits the left transferred model structure, from the classical model structure on $\operatorname{sSet}^{*/}$ by the above adjunction, as shown by Proposition 6.2 in \cite{Goerss2009}. As the wedge product of pointed simplicial sets restricts to a functor
    \begin{equation*}
        \wedge:\operatorname{sSet}_0 \times \operatorname{sSet}_0 \to \operatorname{sSet}_0,
    \end{equation*}
    we see that $\operatorname{sSet}_0$ is closed semi-monoidal with internal hom functor
    \begin{equation*}
        \underline{\operatorname{sSet}}_0(-,-) := \mathscr{R} \underline{\operatorname{sSet}}^{*/}(-,-).
    \end{equation*}
    Note that the unit, $*\sqcup *$, of $\operatorname{sSet}^{*/}$ is not reduced so $\operatorname{sSet}_0$ is not monoidal. 
    \end{example}

    As a consequence of $\operatorname{sSet}_0$ being semi-monoidal and the coreflective Quillen adjunction to $\operatorname{sSet}^{*/}$ we have get the following.
    \begin{cor}
        Any pointed simplicial model category $\mathcal{M}$ is canonically also a $\operatorname{sSet}_0$-module model category.
    \end{cor}
    
    \begin{cor}
        Let $\mathcal{M}$ be a pointed model category. Then its homotopy category can be given the structure of a closed $\operatorname{Ho}(\operatorname{sSet}_0)$-module.
    \end{cor}

We will not further develop the theory here but only note that the standard proofs for monoidal- and module model categories also applies to the non-unital setting. In particular from Section 4.3 in \cite{Hovey1999} we obtain the following statements.
\begin{thm}
    \begin{itemize}
    \item
    Let $(\mathcal{V},\otimes, \underline{\mathcal{V}})$ be a symmetric semi-monoidal model category. Then its homotopy category $\operatorname{Ho}(\mathcal{V})$ has the structure of a semi-monoidal category $(\operatorname{Ho}(\mathcal{V}),\otimes^L,R \underline{\mathcal{V}})$ induced from $\mathcal{V}$.
    \item Let $(\mathcal{C},\overline{\mathcal{C}},\rhd,[-,-])$ be a  $\mathcal{V}$-module model category. Then its homotopy category $\operatorname{Ho}(\mathcal{C})$ has the structure of a closed $\operatorname{Ho}(\mathcal{V})$-module\\ $(\operatorname{Ho}(\mathcal{C}),R\overline{\mathcal{C}},\rhd^L,R[-,-])$ induced from $\mathcal{C}$.
    \item A semi-monoidal Quillen adjunction $(F,U,\varphi,m):\mathcal{V} \to \mathcal{W}$ between symmetric semi-monoidal model categories induces a symmetric semi-monoidal adjunction $(LF,RU,R\varphi,m_{LF}):\operatorname{Ho}(\mathcal{V}) \to \operatorname{Ho}(\mathcal{W})$ of homotopy categories. 
    \item A $\mathcal{V}$-module Quillen adjunction $(F,U,\varphi,m):\mathcal{C} \to \mathcal{D}$ induces a $\operatorname{Ho}(\mathcal{V})$-module adjunction $(LF,RU,R\varphi,m_{LF}):\operatorname{Ho}(\mathcal{C}) \to \operatorname{Ho}(\mathcal{D})$ of homotopy categories.
    \end{itemize}
\end{thm}

%% file: Modulestructure.tex
\section{The semi-module structure of $\DGA_0$ over $\cDGA$}\label{secassocmodule}
    From \cite{Anel2013} we know that $\DGA_0$ admits a closed $\operatorname{coDGA}_0$-module structure. To show that $\DGA_0$ also admits a closed $\cDGA$-semi-module structure we will make use of the same procedure. Indeed, applying the conilpotent radical functor $R^{\operatorname{co}}$ to the internal hom functor and the enrichment functor in the $\operatorname{coDGA}_0$ case would provide the results in the conilpotent case. Nevertheless, for the convenience of the reader we provide categorical proofs of the needed results.
    
    \subsection{The semi-monoidal structure of conilpotent coalgebras}
    It is well known from e.g. \cite{Anel2013} that the category of non-counital dg coalgebras $(\operatorname{coDGA}_0, \otimes)$ is symmetric monoidal. We consider the full subcategory of conilpotent non-counital dg coalgebras $\cDGA$. By \Cref{propconilclosedundertensor} we know that the tensor product functor of conilpotent coalgebras restrict to a functor 
    \begin{equation*}
        \otimes: \cDGA\times \cDGA \to \cDGA,
    \end{equation*}
     thus providing a semi-monoidal structure on $\cDGA$. However $\cDGA$ does not quite form a monoidal category under the tensor product as the monoidal unit $k$ of $\operatorname{coDGA}_0$ is not conilpotent.

    \begin{thm}\label{coDGAclosed}
    The category $(\cDGA,\otimes)$ is closed symmetric semi-monoidal.
\end{thm}
\begin{proof}
    We need to construct an internal hom functor
    \begin{equation*}
        \icDGA(-,-):\left(\cDGA\right)^{\operatorname{op}} \times \cDGA \to \cDGA,
    \end{equation*}
    satisfying that $\icDGA(C_1,-)$ is right adjoint to the tensor product functor $-\otimes C_1$ for every $C_1\in \cDGA$.

    For the case of a conilpotent cofree coalgebra $T^{\operatorname{co}}V$ we have the natural isomorphisms
    \begin{equation*}
    \begin{aligned}
        & \cDGA(C_0\otimes C_1,T^{\operatorname{co}}_0V) \cong \DGVec(C_0\otimes C_1,V) \cong\\ &\DGVec(C_0,\iVec(C_1,V)) \cong  \cDGA(C_0,T^{\operatorname{co}}\iVec(C_1,V)).
        \end{aligned}
    \end{equation*}
    Hence we can define
    \begin{equation*}
        \icDGA(C,T^{\operatorname{co}}V) := T^{\operatorname{co}}\iVec(C,V). 
    \end{equation*}
    In the case of an arbitrary conilpotent coalgebra $D$ we can write $D$ as an equaliser
    \begin{equation*}
        \begin{tikzcd}
            D \arrow[r] & T^{\operatorname{co}}D \arrow[r,shift left=0.5ex] \arrow[r,shift right=0.5ex] & (T^{\operatorname{co}})^2D.
        \end{tikzcd}
    \end{equation*}
    We then define the internal hom $\icDGA(C,D)$ as the equaliser of 
    \begin{equation*}
        \begin{tikzcd}
            \icDGA(C,T^{co}D) \arrow[r,shift left=0.5ex] \arrow[r,shift right=0.5ex] \arrow[d,phantom,"\cong"] & \icDGA(C,(T^{co})^2D) \arrow[d,phantom,"\cong"] \\
            T^{co}\iVec(C,D) \arrow[r,shift left=0.5ex] \arrow[r,shift right=0.5ex] & T^{\operatorname{co}}\iVec(C,T^{co}D)
        \end{tikzcd}
    \end{equation*}
    Since the $\cDGA(C_0,-)$ functor preserves limits this indeed gives the desired natural bijection
    \begin{equation*}
        \cDGA(C_0 \otimes C,D) \cong \cDGA(C_0,\icDGA(C,D)).
    \end{equation*}
\end{proof}

Note that the internal hom of $\cDGA$ is substantially different from the internal hom of $\operatorname{coDGA}_0$. The latter can be constructed analogously using the non-unital cofree functor $T_0$ in place of the conilpotent cofree functor $T^{\operatorname{co}}$. In particular we have that $\icDGA(C,D) \cong R^{\operatorname{co}}\underline{\operatorname{coDGA}}_0(C,D)$ for all conilpotent coalgebras $C$ and $D$.

\begin{remark}
As we have no unit in a semi-monoidal category we have no way of recovering the hom sets from the internal hom. Indeed in our case that information is lost as the only atom of $\icDGA(C,D)$ is $0$, corresponding to the zero morphism. This is also the reason we cannot consider some larger subcategory of $\operatorname{coDGA}_0$ containing $k$ as, even if such a category admits an internal hom, it will never be conilpotent. 
\end{remark}

%% file: DGAadjunctions.tex
\subsection{The semi-module structure of non-unital DG-algebras}
    In addition to the internal hom adjunction for conilpotent coalgebras established in the previous section we will need to establish tensoring and cotensoring adjunctions. 
    
    Our starting point is to consider the convolution algebra functor restricted to conilpotent coalgebras
    \begin{equation*}
        \{-,-\}:\left(\cDGA\right)^{\operatorname{op}} \times \DGA_0 \to \DGA_0.
    \end{equation*}
    We will show that this will be the cotensoring functor of the closed semi-module structure of $\DGA_0$. We construct the enriched hom functor as the left adjoint to the opposite convolution algebra functor. 
    \begin{prop}\label{propenrichmentfunctor}
        There exists a functor
        \begin{equation*}
            \eDGA(-,-):\DGA_0^{\operatorname{op}} \times \DGA_0 \to \cDGA,
        \end{equation*}
        such that $\eDGA(-,B)$ is right adjoint to the opposite convolution algebra functor $\{-,B\}^{\operatorname{op}}$ for each algebra $B\in \DGA_0$. We will refer to $\eDGA$ as the enrichment functor.
    \end{prop}
    \begin{proof}
        For a free algebra $T_0V$ we have natural isomorphisms
        \begin{equation*}
            \begin{aligned}
            &\DGA_0(T_0V,\{C,B\}) \cong \DGVec(V,\iVec(C,B)) \cong \DGVec(V\otimes C,B) \cong\\
            &\DGVec(C,\iVec(V,B)) \cong \cDGA(C,T_0^{\operatorname{co}}\iVec(V,B)),
            \end{aligned}
        \end{equation*}
        so we can define the enriched hom as
        \begin{equation*}
            \eDGA(T_0V,B) := T_0^{\operatorname{co}}\iVec(V,B).
        \end{equation*}
        Given an arbitrary algebra $A$ we can write it as a coequaliser 
        \begin{equation*}
        \begin{tikzcd}
            T^2_0A \arrow[r,shift left=0.5ex] \arrow[r,shift right=0.5ex] & T_0A \arrow[r] & A,
        \end{tikzcd}
    \end{equation*}
        and define $\eDGA(A,B)$ as the equaliser of
        \begin{equation*}
        \begin{tikzcd}
            \eDGA(T_0A,B) \arrow[r,shift left=0.5ex] \arrow[r,shift right=0.5ex] \arrow[d,phantom,"\cong"] & \eDGA(T_0^2A,B) \arrow[d,phantom,"\cong"]\\
           T^{\operatorname{co}}\iVec(A,B) \arrow[r,shift left=0.5ex] \arrow[r,shift right=0.5ex] & T^{\operatorname{co}}\iVec(T_0A,B).
        \end{tikzcd}
    \end{equation*}
    Since $\cDGA(C,-)$ preserves limits we get the desired natural bijection
    \begin{equation*} 
        \DGA_0(A,\{C,B\}) \cong  \cDGA(C,\eDGA(A,B)).
    \end{equation*}
    \end{proof}
    
    We construct what will be the tensoring functor similarly.
    \begin{prop}\label{proptensoringfunctor}
        There exists a functor
        \begin{equation*}
            (-) \rhd (-):\cDGA \times \DGA_0 \to \DGA_0,
        \end{equation*}
        such that $C \rhd (-)$ is left adjoint to the convolution algebra functor $\{C,-\}$ for each coalgebra $C\in \cDGA$. We will refer to $\rhd$ as the tensoring functor.
    \end{prop}
    \begin{proof}
    For a free algebra $T_0V$ we have the natural isomorphisms
    \begin{equation*}
        \begin{aligned}
            &\DGA_0(T_0V,\{C,B\}) \cong  \DGVec(V,\iVec(C,B)) \cong\\
            &\DGVec(C\otimes V,B) \cong \DGA_0(T_0(C\otimes V),B).
        \end{aligned}
    \end{equation*}
         Hence we can define
        \begin{equation*}
            C \rhd T_0V := T_0(C \otimes V).
        \end{equation*}
        Given an arbitrary algebra $A$ we can write it as a coequaliser and define $C\rhd A$ as the the coequaliser of
        \begin{equation*}
        \begin{tikzcd}
            C\rhd T_0^2 A \arrow[r,shift left=0.5ex] \arrow[r,shift right=0.5ex] \arrow[d,phantom,"\cong"]& C \rhd T_0A \arrow[d,phantom,"\cong"]\\
            T_0(C\otimes T_0A) \arrow[r,shift left=0.5ex] \arrow[r,shift right=0.5ex] & T_0(C\otimes A).
        \end{tikzcd}
    \end{equation*}
    Since $\DGA_0(-,B)$ takes colimits to limits we get the desired natural bijection
        \begin{equation*}
            \DGA_0(A,\{C,B\})\cong \DGA_0(C\rhd A,B).
        \end{equation*}
    \end{proof}

    By combining \Cref{propenrichmentfunctor} and \Cref{proptensoringfunctor} we also get a third adjunction between the tensoring and the enrichement functor.
    \begin{cor}\label{cortensenrich}
        The tensoring $(-)\rhd A$ is left adjoint to the enrichment functor $\eDGA(A,-)$ for each algebra $A\in \DGA_0$.
    \end{cor}
    
    In summary we have adjunctions,
    \begin{equation*}
        \begin{tikzcd}
            \cDGA \arrow[r,shift left=1.1ex,"{\{-,B\}^{\operatorname{op}}}"] \arrow[r,phantom,"\perp"] & \DGA^{\operatorname{op}} \arrow[l,shift left=1.1ex,"{\eDGA(-,B)}"],
        \end{tikzcd}
    \end{equation*}
    
    \begin{equation*}
        \begin{tikzcd}
            \DGA_0 \arrow[r,shift left=1.1ex,"C\rhd (-)"] \arrow[r,phantom,"\perp"] & \DGA_0 \arrow[l,shift left=1.1ex,"{\{C,-\}}"],
        \end{tikzcd}
    \end{equation*}
    and
\begin{equation*}
        \begin{tikzcd}
            \cDGA \arrow[r,shift left=1.1ex,"(-)\rhd A"] \arrow[r,phantom,"\perp"] & \DGA_0 \arrow[l,shift left=1.1ex,"{\eDGA(A,-)}"].
        \end{tikzcd}
    \end{equation*}
    These are analogous to those shown in \cite{Anel2013} for the non-conilpotent case. In particular the tensoring and cotensoring functors are the same as in the non-conilpotent case while the enrichment functor differs. 

%% file: Measurings.tex
    \subsubsection{Measurings and coherence}
    To give $\DGA_0$ the structure of a module category we also need it to satisfy the coherence axiom. To show this we will use the concept of measurings developed in \cite{Sweedler1969}. We will briefly repeat the definition of measurings and some properties we will make use of, while referring the reader to \cite{Anel2013} for a more extensive coverage. 
    
    \begin{definition}
        Let $A,B$ be non-unital dg-algebras and $C$ a non-counital dg-coalgebra. We say that a dg-linear morphism
        $f:C\otimes A \to B$ is a measuring if the adjoint morphism,
        \begin{equation*}
            A \to \{C,B\},
        \end{equation*}
        is a morphism of non-unital dg-algebras. 
    \end{definition}
    As an immediate consequence of the definition we have the following.
    \begin{prop}\label{propmeasuringfunctor}
        Let $v:C\otimes A \to B$ be a measuring, $f:A'\to A$ and $g:B\to B'$ be algebra maps and $h:C'\to C$ a coalgebra map. Then the composition
        \begin{equation*}
            g\circ v \circ (h\otimes f):C' \otimes A' \to B'
        \end{equation*}
        is also a measuring.
    \end{prop}

    \begin{prop}\label{propextmeasuring}
        Let $C$ be a non-counital dg-coalgebra, $A$ a non-unital dg-algebra and $V$ a dg-vector space. Then a dg-linear map
        \begin{equation*}
            f:C\otimes V \to A,
        \end{equation*}
        extends uniquely to a measuring
        \begin{equation*}
            f:C\otimes T_0V \to A.
        \end{equation*}
    \end{prop}
    \begin{proof}
        By the tensor-hom adjunction for dg-vector spaces, the free forgetful adjunction for dg-algebras, and the definition of measurings we have natural equivalences
        \begin{equation*}
            \begin{aligned}
                &\DGVec(C\otimes V,A) \cong \DGVec(V,\iVec(C,A)) \cong\\
                &\DGA_0(T_0V,\{C,V\}) \cong \mathcal{M}(C,T_0V,A).
            \end{aligned}
        \end{equation*}
    \end{proof}

    We denote by $\mathcal{M}(C,A,B)$ the set of measurings from $C\otimes A \to B$. By \Cref{propmeasuringfunctor} this assignment extends to a functor
    \begin{equation*}
        \mathcal{M}(-,-;-):\left(\operatorname{coDGA}_0\right)^{\operatorname{op}}\times \DGA_0^{\operatorname{op}} \times \DGA_0 \to \text{Set},
    \end{equation*}
    which we will refer to as the measurement functor. We will consider the measurement functor restricted to conilpotent coalgebras,
        \begin{equation*}
            \mathcal{M}(-,-;-):\left(\cDGA\right)^{\operatorname{op}}\times \DGA_0^{\operatorname{op}} \times \DGA_0 \to \text{Set},
    \end{equation*}
    which we will refer to as the restricted measurement functor.
    The measurement functor is representable in each variable, which is shown in Section 4.1 of \cite{Anel2013}. The argument to show that it is also representable in the restricted case is similar which we briefly repeat here.
    \begin{prop}\label{propmeasuringrep}
        The restricted measurement functor $\mathcal{M}$ is represented in each variable.
    \end{prop}
    \begin{proof}
        By definition of measuring the restricted measurement functor is represented in the second variable by
        \begin{equation*}
        \left(\{C,B\},\epsilon:C\otimes \{C,B\} \to B\right),
    \end{equation*}
    where $\epsilon$ is the counit of the tensor-hom adjunction for dg-vector spaces.
    The representability in the remaining variables now follows from the tensored and cotensored adjunctions constructed in \Cref{propenrichmentfunctor} and \Cref{proptensoringfunctor}. That is we have isomorphisms,
        \begin{equation*}
            \begin{aligned}
                &\mathcal{M}(C,A,B) \cong\\
                &\DGA_0(A,\{C,B\}) \cong \DGA_0(C\rhd A,B) \cong \cDGA(C,\eDGA(A,B)),
            \end{aligned}
    \end{equation*}
    natural in each variable. 
    \end{proof}

    \begin{remark}
        Explicitly we have that the restricted measuring functor is represented in the third variable by
    \begin{equation*}
        \left(C\rhd A,u:C\otimes A \to C\rhd A\right),
    \end{equation*}
    where $u$ is the measuring induced from the inclusion map
    \begin{equation*}
        i:C\otimes V \to T_0(C \otimes V),
    \end{equation*}
     by \Cref{propextmeasuring}. 
     
    Similarly, the restricted measuring functor is represented in the first variable by
    \begin{equation*}
        (\eDGA(A,B),\mathbf{ev}:\eDGA(A,B) \otimes A \to B).
    \end{equation*}
    Here the evaluation map $\mathbf{ev}$ is the composition
    \begin{equation*}
                \eDGA(A,B) \otimes A \xrightarrow{u} \eDGA(A,B) \rhd A \xrightarrow{\epsilon_\rhd} B,
    \end{equation*}
    with $\epsilon_\rhd$ denoting the counit of the tensoring-enrichment functor adjunction of \Cref{cortensenrich}. 
    \end{remark}
    
    We can now proceed with showing the semi-module structure of $(\DGA_0,\rhd)$. 
    \begin{lemma}\label{lemmaassoc}
        There exists a unique natural isomorphism
        \begin{equation*}
            a:(C\otimes D) \rhd A \to C\rhd (D\rhd A),
        \end{equation*}
        such that
        \begin{equation*}
                \begin{tikzcd}
                (C\otimes D) \otimes A \arrow[r,"\alpha"] \arrow[d,"u"'] & C\otimes (D\otimes A) \arrow[d,"u\circ (1\otimes u)"]\\
                (C\otimes D) \rhd A \arrow[r,dashed,"a"] & C\rhd (D\rhd A)
                \end{tikzcd}
            \end{equation*}
            commutes.
    \end{lemma}
    \begin{proof}
        This follows from the universal property of $u$ i.e. for every measuring $v:C\otimes A \to B$ there exists a unique algebra morphism $f:C\rhd A \to B$ such that
        \begin{equation*}
                \begin{tikzcd}
                    C\otimes A \arrow[r,"v"] \arrow[d,"u"']& B\\
                    C\rhd A \arrow[ur,dashed,"\exists !f"']
                \end{tikzcd}
            \end{equation*}
            commutes.
    \end{proof}

    We now have sufficient background to prove the main result of this section. The reader should however note that the following result also follows as a straightforward corollary from Theorem 4.1.18 in \cite{Anel2013}.
    \begin{thm}\label{DGAclosedmodule}
        $(\DGA_0,\rhd,\eDGA,\{-,-\})$ is a closed $\cDGA$-module.
    \end{thm}
    \begin{proof}
    It remains to show that the coherence axiom is satisfied for the associator $a$ constructed in the \cref{lemmaassoc}. Consider the diagram
    \begin{equation*}
    \adjustbox{scale=0.6,center}{
    \begin{tikzcd}[column sep=tiny]
        & ((C\otimes D)\otimes E)\otimes A \arrow[rrrr,"\alpha"] \arrow[dl,"\alpha {\otimes} 1"'] \arrow[dd,shift right=2.5ex] &&&& (C\otimes D) \otimes (E\otimes A) \arrow[dl,"\alpha"] \arrow[dd]]\\
        (C\otimes (D\otimes E))\otimes A \arrow[rr,crossing over,"\alpha"] \arrow[dd] && C\otimes ((D\otimes E)\otimes A)\arrow[rr,"1\otimes \alpha"] && C\otimes (D\otimes (E\otimes A))
        \\
        & ((C\otimes D) \otimes E)\rhd A \arrow[rrrr,"a"] \arrow[dl,"\alpha \rhd 1"'] &&&& (C\otimes D)\rhd (E\rhd A) \arrow[dl,"a"]\\
        (C\otimes (D\otimes E)) \rhd A \arrow[rr,"a"] && C\rhd ((D\otimes E)\rhd A)\arrow[rr,"1\rhd a"] \arrow[from=uu,crossing over] && C \rhd (D\rhd (E\rhd A) \arrow[from=uu,crossing over]
    \end{tikzcd}
    }
    \end{equation*}
    where the vertical arrows consist of the universal element $u$ applied as demanded by the diagram. We conclude that every vertical face commutes by \Cref{lemmaassoc} and that the top face commutes by the monoidal structure of the tensor product $\otimes$ on vector spaces. Thus after precomposition with the morphism $u:((C\otimes D)\otimes E) \otimes A \to ((C\otimes D)\otimes E)\rhd A$ the bottom face commutes. But by the universal property of $u$ we have that $u$ is right-cancellative on algebra morphisms. Hence the bottom face commutes.
    \end{proof}

%% file: McDGA.tex
\section{Semi-monoidal model structure on $\cDGA$}\label{seccoDGAmodel}  
    For showing that $\cDGA$ is a semi-monoidal model category, we will first establish that the tensor product functor of $\cDGA$ preserves (acyclic) cofibrations in each variable separately.
    \begin{lemma}\label{lemmatensorpres}
        The tensor product functor
        \begin{equation*}
            \otimes: \cDGA \times \cDGA \to \cDGA,
        \end{equation*}
        preserves cofibrations and weak equivalences in each variable separately.
    \end{lemma}
    \begin{proof}
         We first note that the forgetful functor $U:\cDGA \to \DGVec$ commutes with the tensor product and preserves cofibrations. The preservation of cofibrations under the tensor product of $\cDGA$ then follows from the $\DGVec$ case. 
         
         It remains to show the preservation of weak equivalences. Since the class of weak equivalences is the closure of the class of filtered quasi-isomorphism under the 2 out of 3 property it suffices to show the preservation of filtered quasi-isomorphisms. Thus let $f:C \to D$ be a filtered quasi-isomorphism and let $E$ be a conilpotent coalgebra. By assumption there exist admissible filtrations $F_C$ and $F_D$, of $C$ and $D$ respectively, such that
        \begin{equation*}
            \text{gr}^if:\text{gr}_C^i F_C \to \text{gr}_D^i F_D
        \end{equation*}
        is a quasi-isomorphism in each degree. We define filtrations $F_{C\otimes E} := F_C \otimes E$ and $F_{D\otimes E} := F_D \otimes E$ and note that they are admissible. Further noting that $\text{gr}_{C\otimes E} \cong \text{gr}_C \otimes E$ we get an induced quasi-isomorphism,
        \begin{equation*}
            \begin{tikzcd}
                \text{gr}^i F_{C\otimes E} \arrow[r,"\sim"] \arrow[d,phantom,"\cong"] & \text{gr}^i F_{D\otimes E} \arrow[d,phantom,"\cong"] \\
                \text{gr}^i F_{C} \otimes E \arrow[r,"\sim","\text{gr}^i f \otimes E"'] & \text{gr}^i F_{D} \otimes E,
            \end{tikzcd}
        \end{equation*}
        using that the tensor product of dg-vector spaces preserves quasi-isomorphism. Thus $f\otimes E$ is a filtered quasi-isomorphism.
    \end{proof}
    
    \begin{thm}\label{coDGAmonoidalmodel}
    $(\cDGA,\otimes)$ is a symmetric semi-monoidal model category. 
\end{thm}
\begin{proof}
We have to show that $\otimes$ is a Quillen bifunctor. Let $i:C\to C'$ be a cofibration and $j:D\to D'$ an (acyclic) cofibration in $\cDGA$. The relevant pushout diagram is 
\begin{equation*}
        \begin{tikzcd}
            C\otimes D \arrow[r,hookrightarrow,"{\Id_C \otimes j}","\sim"'] \arrow[d,hookrightarrow,"{i \otimes \Id_D}"'] & C\otimes D' \arrow[d,hookrightarrow,"\iota_2"] \arrow[dddrr,hookrightarrow,bend left,"{i \otimes \Id_{D'}}"]\\
            C'\otimes D \arrow[r,hookrightarrow,"\iota_1"',"\sim"] \arrow[ddrrr,hookrightarrow,bend right,swap,"{\Id_{C'} \otimes j}","\sim"'] & (C'\otimes D) \underset{C\otimes D}{\coprod} (C\otimes D') \arrow[ddrr,dashed,"\exists !"]\\\\
            & & & C'\otimes D'
        \end{tikzcd}
    \end{equation*}
    where we use that the (acyclic) cofibrations are closed under pullback, and \Cref{lemmatensorpres}. 
    
    We see that in the acyclic case we get that the pushout map is a weak equivalence by the 2 out of 3 property. That the pushout map is injective follows from the dg-vector space case as colimits and cofibrations are preserved by the forgetful functor to $\DGVec$, which furthermore commutes with the tensor product.
\end{proof}

%% file: MDGA.tex
\section{Homotopical enrichment of $\DGA_0$}\label{secDGAmodel}
    To show that $\DGA_0$ is a model $\mathcal{\cDGA}$ category we need to show that the tensoring functor $\rhd$ is a Quillen bifunctor. However since cofibrations in $\cDGA$ and fibrations in $\DGA_0$ are particularly easy to work with, we will make use of the equivalent condition for the cotensoring functor $\{-,-\}$. That is we  will show that for every cofibration $i:C\to C'$ in $\cDGA$ and every fibration $j:B\to B'$ in $\DGA_0$ the induced map
    \begin{equation*}
        \{A',B\} \to \{A,B\} \times_{\{A,B'\}} \{A',B'\},
    \end{equation*}
    is a cofibration, which furthermore is acyclic if either $i$ or $j$ is. 
\begin{lemma}\label{convpres(co)fib}
    The convolution algebra functor 
    \begin{equation*}
    \{-,-\}:\left(\cDGA\right)^{\operatorname{op}} \times \DGA_0 \to \DGA_0,
    \end{equation*}
    takes (acyclic) cofibrations in the first variable and (acyclic) fibrations in the second variable to (acyclic) fibrations separately. 
\end{lemma}
\begin{proof}
    That we get fibrations in either case is immediate. For the second part we show the stronger statement that the convolution algebra functor preserves quasi-isomorphisms. This is sufficient as every weak equivalence in $\cDGA$ by necessity is also a quasi-isomorphism. Next note that the forgetful functor to $\DGVec$ preserves quasi-isomorphisms and commutes with the cotensoring functor. That is the convolution algebra functor is taken to the internal hom of dg-vector spaces. As the internal hom of dg-vector spaces is exact the preservation of quasi-isomorphisms follows.
\end{proof}

\begin{thm}\label{DGAmonoidalmodel}
    $(\DGA_0,\overline{\DGA}_0,\rhd,\{-,-\})$ is a $\cDGA$-model category.
\end{thm}
\begin{proof}
Let $i:C\hookrightarrow C'$ be a cofibration in $\cDGA$ and $j:A\twoheadrightarrow A'$ a (acyclic) fibration in $\DGA_0$. The relevant pullback diagram is
\begin{equation*}
    \begin{tikzcd}
        \{C',A\} \arrow[drr,twoheadrightarrow,bend left,"{\{C',j\}}"] \arrow[ddr,twoheadrightarrow,bend right,"{\{i,A\}}"'] \arrow[dr,dashed,"\exists !"]\\
        & \{C,A\} \underset{\{C,A'\}}{\prod} \{C',A'\} \arrow[d,twoheadrightarrow] \arrow[r,twoheadrightarrow] & \{C',A'\} \arrow[d,twoheadrightarrow,"{\{i,A'\}}"]\\
        & \{C,A\} \arrow[r,twoheadrightarrow,"{\{C,j\}}"',"\sim"] & \{C,A'\}
    \end{tikzcd}
\end{equation*}
where we use that (acyclic) fibrations are closed on pullback, and \Cref{convpres(co)fib}. It follows from the 2 out of 3 property that the induced morphism is a weak equivalence. Similarly had we instead started assuming that $i$ was acyclic we would've reached the same conclusion. That the pullback product map is surjective follows from the dg-vector space case as colimits and cofibrations are preserved by the forgetful functor to $\DGVec$ which furthermore commutes with the cotensoring.
\end{proof}

%% file: DGLie.tex
 \section{The com-Lie case}
We have so far, in \cref{secassocmodule,seccoDGAmodel,secDGAmodel}, shown the homotopical enrichment corresponding to the case of associative dg Koszul duality. We will now similarly proceed with homotopical enrichment in the com-Lie case of Koszul duality. That is we will show that $\DGLA$ is a semi-module model category over $\cCDGA$. The procedure will be completely analogous to the associative case.

We first note that $(\cCDGA,\otimes)$ is symmetric semi-monoidal, which follows from the $\cDGA$ case. We have that it is also closed by the following.
 \begin{thm}\label{coCDGAclosed}
    The category $(\cCDGA,\otimes)$ is closed symmetric semi-monoidal.
\end{thm}
\begin{proof}
    For a cocommutative cofree coalgebra $S^{\operatorname{co}}V$ we define the internal hom $\icDGA(C,S^{\operatorname{co}}V)$ as
    \begin{equation*}
        \icCDGA(C,S^{\operatorname{co}}V) := S^{\operatorname{co}}\iVec(C,V).
    \end{equation*}
    For an arbitrary cocommutative coalgebra $D$ we can write it as an equaliser
    \begin{equation*}
        \begin{tikzcd}
            D \arrow[r] & S^{\operatorname{co}}D \arrow[r,shift left=0.5ex] \arrow[r,shift right=0.5ex] & (S^{\operatorname{co}})^2D.
        \end{tikzcd}
    \end{equation*}
    We then define the internal hom functor $\icCDGA(C,D)$ as the equaliser of
    \begin{equation*}
        \begin{tikzcd}
            \icCDGA(C,S^{\operatorname{co}}D) \arrow[r,shift left=0.5ex] \arrow[r,shift right=0.5ex] \arrow[d,phantom,"\cong"] & \icCDGA(C,(S^{\operatorname{co}})^2D) \arrow[d,phantom,"\cong"] \\
            S^{\operatorname{co}}\iVec(C,D) \arrow[r,shift left=0.5ex] \arrow[r,shift right=0.5ex] & S^{\operatorname{co}}\iVec(C,S^{\operatorname{co}}D).
        \end{tikzcd}
    \end{equation*}
    That this functor is right adjoint to the tensoring functor now follows by the same argument as in the proof of \Cref{coDGAclosed}.
\end{proof}

We will next establish the $\cCDGA$-module structure of $\DGLA$.

\begin{prop}\label{propdglaenrichmentfunctor}
        There exists a functor
        \begin{equation*}
            \overline{\DGLA}(-,-):\DGLA^{\operatorname{op}} \times \DGLA \to \cCDGA,
        \end{equation*}
        such that $\overline{\DGLA}(-,\mathfrak{h})$ is right adjoint to the opposite convolution algebra functor $\{-,\mathfrak{h}\}^{\operatorname{op}}$ for each Lie algebra $\mathfrak{h}\in \DGLA$. We will refer to $\overline{\DGLA}$ as the enrichment functor.
    \end{prop}
    \begin{proof}
        For a free Lie algebra $T_{\operatorname{Lie}}V$ we have the natural isomorphism
        \begin{equation*}
            \begin{aligned}
            &\DGLA(T_{\operatorname{Lie}}V,\{C,\mathfrak{h}\}) \cong \DGVec(V,\iVec(C,\mathfrak{h})) \cong \DGVec(V\otimes C,\mathfrak{h}) \cong\\
            &\DGVec(C,\iVec(V,\mathfrak{h})) \cong \cCDGA(C,S^{\operatorname{co}}(\iVec(V,\mathfrak{h})),
            \end{aligned}
        \end{equation*}
        so we can define the enriched hom functor as
        \begin{equation*}
            \underline{\DGLA}(T_{\operatorname{Lie}}V,\mathfrak{h}) := S^{\operatorname{co}}\iVec(V,\mathfrak{h}).
        \end{equation*}
        Given an arbitrary $\mathfrak{g} \in \DGLA$ we write it as a coequaliser
        \begin{equation*}
        \begin{tikzcd}
            T^2_{\operatorname{Lie}}\mathfrak{g} \arrow[r,shift left=0.5ex] \arrow[r,shift right=0.5ex] & T_{\operatorname{Lie}}\mathfrak{g}\arrow[r] & \mathfrak{g}.
        \end{tikzcd}
    \end{equation*}
        We then define $\overline{\DGLA}(\mathfrak{g},\mathfrak{h})$ as the equaliser of
        \begin{equation*}
        \begin{tikzcd}
            \overline{\DGLA}(T_{\operatorname{Lie}}^2\mathfrak{g},\mathfrak{h}) \arrow[r,shift left=0.5ex] \arrow[r,shift right=0.5ex] \arrow[d,phantom,"\cong"] & \overline{\DGLA}(T_{\operatorname{Lie}}\mathfrak{g},\mathfrak{h}) \arrow[d,phantom,"\cong"]\\
           S^{\operatorname{co}}\iVec(T_{\operatorname{Lie}}\mathfrak{g},\mathfrak{h}) \arrow[r,shift left=0.5ex] \arrow[r,shift right=0.5ex] & S^{\operatorname{co}}\iVec(\mathfrak{g},\mathfrak{h}).
        \end{tikzcd}
    \end{equation*}
    Since $\cCDGA(C,-)$ preserves limits we get the desired natural bijection
    \begin{equation*} 
        \DGLA(\mathfrak{g},\{C,\mathfrak{h}\}) \cong  \cCDGA(C,\overline{\DGLA}(\mathfrak{g},\mathfrak{h})).
    \end{equation*}
    \end{proof}

    We construct the tensoring functor similarly.
    \begin{prop}\label{propdglatensoringfunctor}
        There exists a functor
        \begin{equation*}
            (-) \rhd (-):\cCDGA \times \DGLA \to \DGLA,
        \end{equation*}
        such that $C \rhd (-)$ is left adjoint to the convolution algebra functor $\{C,-\}$ for each coalgebra $C\in \cCDGA$. We will refer to $\rhd$ as the tensoring functor.
    \end{prop}
    \begin{proof}
    For a free Lie algebra $T_{\operatorname{Lie}}V$ we have the natural isomorphism
    \begin{equation*}
        \begin{aligned}
            &\DGLA(T_{\operatorname{Lie}}V,\{C,\mathfrak{h}\}) \cong  \DGVec(V,\{C,\mathfrak{h}\})\\ \cong
            &\DGVec(C\otimes V,\mathfrak{h}) \cong \DGLA(T_{\operatorname{Lie}}(C\otimes V),\mathfrak{h}).
        \end{aligned}
    \end{equation*}
         Hence we define
        \begin{equation*}
            C \rhd T_{\operatorname{Lie}}V := T_{\operatorname{Lie}}(C \otimes V).
        \end{equation*}
        Given an arbitrary Lie algebra $\mathfrak{g}$, we write it as a coequaliser and define $C\rhd \mathfrak{g}$ the coequaliser of
        \begin{equation*}
        \begin{tikzcd}
            C\rhd T_{\operatorname{Lie}}^2 \mathfrak{g} \arrow[r,shift left=0.5ex] \arrow[r,shift right=0.5ex] \arrow[d,phantom,"\cong"]& C \rhd T_{\operatorname{Lie}}\mathfrak{g} \arrow[d,phantom,"\cong"]\\
            T_{\operatorname{Lie}}(C\otimes_k T_{\operatorname{Lie}}\mathfrak{g}) \arrow[r,shift left=0.5ex] \arrow[r,shift right=0.5ex] & T_{\operatorname{Lie}}(C\otimes_k \mathfrak{g}).
        \end{tikzcd}
    \end{equation*}
    Since $\DGLA(-,\mathfrak{h})$ takes colimits to limits we get the desired natural bijection
        \begin{equation*}
            \DGLA(\mathfrak{g},\{C,\mathfrak{h}\})\cong \DGLA(C\rhd \mathfrak{g},\mathfrak{h}).
        \end{equation*}
    \end{proof}

    By combining \Cref{propdglaenrichmentfunctor} and \Cref{propdglatensoringfunctor} we also get a third adjunction between the tensoring and the enrichement functor.
    \begin{cor}
        The tensoring $(-)\rhd \mathfrak{h}$ is left adjoint to the enrichment functor $\overline{\DGLA}(\mathfrak{h},-)$ for each Lie algebra $\mathfrak{h} \in \DGLA$.
    \end{cor}
    In summary we have established adjunctions
\begin{equation*}
        \begin{tikzcd}
            \cCDGA \arrow[r,shift left=1.1ex,"{\{-,\mathfrak{h}\}^{\operatorname{op}}}"] \arrow[r,phantom,"\perp"] & \DGLA^{\operatorname{op}} \arrow[l,shift left=1.1ex,"{\overline{\DGLA}(-,\mathfrak{h})}"],
        \end{tikzcd}
    \end{equation*}
    
    \begin{equation*}
        \begin{tikzcd}
            \DGLA \arrow[r,shift left=1.1ex,"C\rhd (-)"] \arrow[r,phantom,"\perp"] & \DGLA \arrow[l,shift left=1.1ex,"{\{C,-\}}"],
        \end{tikzcd}
    \end{equation*}
    and
\begin{equation*}
        \begin{tikzcd}
            \cCDGA \arrow[r,shift left=1.1ex,"(-)\rhd A"] \arrow[r,phantom,"\perp"] & \DGLA \arrow[l,shift left=1.1ex,"{\overline{\DGLA}(\mathfrak{h},-)}"].
        \end{tikzcd}
    \end{equation*}

    As in the associative case we make use of the concept of measurings to show the coherence axiom for the module structure of $\DGLA$. Adapted to the Lie algebra case the definition becomes the following.
    \begin{definition}
        Let $A,B$ be dg-Lie algebras and $C$ a cocommutative non-counital dg-coalgebra. We say that a dg-linear morphism
        $f:C\otimes A \to B$ is a measuring if the adjoint morphism
        \begin{equation*}
            A \to \{C,B\},
        \end{equation*}
        is a morphism of Lie algebras. 
    \end{definition}
    As in the associative case we will denote by $\mathcal{M}(C,A,B)$ the set of measurings from $C\otimes A \to B$ and note this extends to a functor. 
    \begin{prop}
        The restricted measurement functor
        \begin{equation*}
            \mathcal{M}(-,-;-):\left(\cCDGA\right)^{\operatorname{op}}\times \DGLA^{\operatorname{op}} \times \DGLA \to \operatorname{Set},
    \end{equation*}
        is represented in each variable. 
    \end{prop}
    \begin{proof}
        Same as the proof of \Cref{propmeasuringrep}.
    \end{proof}

    \begin{thm}
        $(\DGLA,\rhd,\overline{\DGLA},\{-,-\})$ is a closed $\cCDGA$-module.
    \end{thm}
    \begin{proof}
        Same as the proof of \Cref{lemmaassoc} and \Cref{DGAclosedmodule}.
    \end{proof}

We are now ready to proceed with the homotopical perspective, which also is shown fully analogously to the associative case. 
\begin{prop}
    $(\cCDGA,\otimes)$ is a symmetric semi-monoidal model category. 
\end{prop}
\begin{proof}
    Same as the proof of  \Cref{coDGAmonoidalmodel}.
\end{proof}

\begin{lemma}\label{convpres(co)fibLie}
    The convolution algebra functor 
    \begin{equation*}
    \{-,-\}:\left(\cCDGA\right)^{\operatorname{op}} \times \DGLA \to \DGLA,
    \end{equation*}
    takes (acyclic) cofibrations in the first variable and (acyclic) fibrations in the second variable to (acyclic) fibrations separately.
\end{lemma}
\begin{proof}
    Same as the proof of \Cref{convpres(co)fib}.
\end{proof}

\begin{thm}
    $(\DGLA,\overline{\DGLA}(-,-),\rhd,\{-,-\})$ is a
    $\cCDGA$-\\module model category.
\end{thm}
\begin{proof}
    Using \Cref{convpres(co)fibLie} we see that the same as the proof of \Cref{DGAmonoidalmodel} goes through.    
\end{proof}

%% file: Koszulduality.tex
\section{Semi-enriched Koszul duality} 
Having established the the semi-module category of structures of $\DGA_0$ and $\DGLA$ we will now return our attention to Koszul duality. Our main aim is to establish \Cref{thm1} and \Cref{thmLie}. That is we will show that the Quillen equivalences in both the associative and the com-Lie case of Koszul duality upgrades to semi-module Quillen equivalences. 

As pointed out in \cite{Anel2013} the bar and cobar constructions are directly related to the constructed enrichment functor and the tensoring functor respectively. Specifically we have the following result.
\begin{prop}\label{proprhdomega}
    There exists natural isomorphisms
    \begin{equation*}
        \Omega C \cong C \rhd \mathbf{mc} \quad \text{ and } \quad  BA \cong \eDGA (\mathbf{mc},A),
\end{equation*}
where $\mathbf{mc}$ is the universal Maurer-Cartan algebra.
\end{prop}
\begin{proof}
    By the tensoring adjunction and representability of the\\ Maurer-Cartan functor $\operatorname{MC}$ we have natural isomorphisms
    \begin{equation*}
        \DGA(C \rhd \mathbf{mc}, A) \cong \DGA(\mathbf{mc},\{C,A\}) \cong \operatorname{MC}(\{C,A\}).
    \end{equation*}
    The statement now follows from the bar-cobar adjunction
    \begin{equation*}
        \DGA_0(\Omega C,A) \cong \operatorname{MC}(\{C,A\}) \cong \cDGA(C,BA).
    \end{equation*}
\end{proof}

Similarly for the $\DGLA$ case we have the analogous result.
\begin{prop}\label{proprhdomegaLie}
    There exists natural isomorphisms
    \begin{equation*}
        \Omega C \cong C \rhd \mathbf{mc}_{\operatorname{Lie}} \quad \text{ and } \quad B A \cong \DGLA(\mathbf{mc}_{\operatorname{Lie}},A),
\end{equation*}
where $\mathbf{mc}_{\operatorname{Lie}}$ is the universal Maurer-Cartan Lie algebra.
\end{prop}
\begin{proof}
    By the tensoring adjunction and representability of the\\
    Maurer-Cartan functor $\operatorname{MC}_{\operatorname{Lie}}$ we have natural isomorphisms
    \begin{equation*}
        \DGLA(C \rhd \mathbf{mc}_{\operatorname{Lie}}, \mathfrak{g}) \cong \DGLA(\mathbf{mc}_{\operatorname{Lie}},\{C,\mathbf{g}\}) \cong \operatorname{MC}(\{C,\mathbf{g}\}).
    \end{equation*}
    The statement now follows from the bar-cobar adjunction
    \begin{equation*}
        \DGLA(\Omega C,\mathfrak{g}) \cong \operatorname{MC}_{\operatorname{Lie}}(\{C,\mathfrak{g}\}) \cong \cCDGA(C,B\mathfrak{g}).
    \end{equation*}
\end{proof}
As a consequence of \Cref{proprhdomega} we see that the bar-cobar adjunction, in the associative case, upgrades to a $\cDGA$-module Quillen equivalence. Similarly \Cref{proprhdomegaLie} implies that the bar-cobar adjunction, in the com-Lie case, upgrades to a $\cCDGA$-module Quillen equivalence. As a consequence, we have established \Cref{thm1} and \Cref{thmLie}.

\begin{remark}
    Note that it also follows that the bar construction $B$, in both the associative and the com-Lie case, is a quasi-strong semi-module functor. Explicitly the weak equivalence is given by
    \begin{equation*}
        C \otimes BA \xrightarrow{\sim} B\Omega (C\otimes BA) \xrightarrow{\cong} B(C\rhd \Omega BA) \xrightarrow{\sim} B(C\rhd A),
    \end{equation*}
    here in the notation of the associative case. 
\end{remark}

\begin{remark}
    It may seem that our results should generalise to the operadic context of Koszul duality. This is however not the case as can be seen from considering a third case of Koszul duality between the category of conilpotent dg-Lie algebra, $\cDGLA$, and the category of commutative non-unital dg-algebras $\operatorname{cDGA}_0$. This case was established in \cite{Lazarev2015} and shows that there is a Quillen equivalence
    \begin{equation*}
    \begin{tikzcd}
        \cDGLA\arrow[r,shift left=1.5ex,"\Omega"] \arrow[r,phantom,"\perp"] & \operatorname{cDGA}_0 \arrow[l,shift left=1.5ex,"B"].
    \end{tikzcd}
    \end{equation*}
    
    The problem here, in regards to extending our results, is that $\cDGLA$ does not have the notion of a tensor product.  Instead one could consider the monoidal structure given by the direct product, which indeed gives a closed monoidal structure on $\cDGLA$ albeit with a quite different internal hom functor from the other cases. However one quickly runs into trouble with defining a $\cDGLA$-module structure for $\operatorname{cDGA}_0$ as we don't have the concept of a convolution algebra or the Sweedler theory adjunctions to rely on.
\end{remark}